\documentclass[11pt]{article}

\usepackage{dsfont}

\usepackage{psfrag}

\usepackage{calrsfs}
\usepackage{mathrsfs}
\usepackage{pdfsync}
\usepackage{amsmath, amsthm, amssymb, amsfonts} 
\usepackage[shortlabels]{enumitem}

\usepackage{url}
\usepackage{float}

\usepackage[usenames]{color}

\usepackage{tikz,fullpage}

\usetikzlibrary{arrows,decorations.pathmorphing,backgrounds,positioning,fit,automata}
\usepackage{pgfplots}
\pgfplotsset{compat=1.15}

\usepackage{indentfirst,calc,euscript}
\usepackage{setspace}
\usepackage[reals]{layout}
\usepackage{xr}
\usepackage{amscd}

\usepackage{sgame}
\usepackage{subfigure}

\usepackage[colorlinks=true,breaklinks=true,bookmarks=true,urlcolor=blue,
     citecolor=blue,linkcolor=blue,bookmarksopen=false,draft=false]{hyperref}

\usepackage[capitalize]{cleveref}
\newcommand{\CRef}[1]{{\color{blue}\hyperref[#1]{\cref{#1}}}}

\newtheorem{theorem}{Theorem}
\newtheorem*{theorem*}{Theorem}

\newtheorem{lemma}[theorem]{Lemma}

\newtheorem{proposition}[theorem]{Proposition}
\newtheorem*{proposition*}{Proposition}

\theoremstyle{definition}

\newtheorem{definition}[theorem]{Definition}

\newtheorem{remark}[theorem]{Remark}

\newtheorem*{notation*}{Notation}

\numberwithin{equation}{section}

\numberwithin{theorem}{section}

\newcommand{\m}{\mathbb}

\newcommand{\Mathbb}[1]{\mathbb{#1}}

\newcommand{\E}{\Mathbb{E}}

\newcommand{\N}{\Mathbb{N}}

\renewcommand{\P}{\Mathbb{P}}

\newcommand{\R}{\Mathbb{R}}

\renewcommand{\phi}{\varphi}

\newcommand{\eps}{\varepsilon}

\newcommand{\ttu}[1]{\underset{#1}{\longrightarrow}}
\newcommand{\ttn}{\ttu{n \to +\infty}}

\usepackage{centernot}

\newcommand{\Lim}{\lim\limits}

\newcommand{\Sup}{\sup\limits}

\usepackage{stmaryrd}
\newcommand{\Intn}[2]{\llbracket #1 ; #2 \rrbracket}

\newcommand{\Sum}[2]{\sum\limits_{#1}^{#2}}

\newcommand{\pinf}{{+\infty}}

\newcommand{\f}[2]{\frac{#1}{#2}}

\newcommand{\Set}[2]{\left\{\,#1 \;\middle|\; #2 \,\right\}}

\newcommand{\pa}[1]{\left( #1 \right)}
\newcommand{\pac}[1]{\left[ #1 \right]}

\newcommand{\abs}[1]{\left| #1 \right|}
\newcommand{\floor}[1]{\left\lfloor #1 \right\rfloor}
\newcommand{\ceil}[1]{\left\lceil #1 \right\rceil}

\newcommand{\intoo}[2]{\,\left]#1,#2\right[}
\newcommand{\intof}[2]{\,\left]#1,#2\right]}

\newcommand{\qandq}{\quad \text{ and } \quad}

\newcommand{\norm}[1]{\abs{\abs{#1}}}

\newcommand{\fun}[4]{\begin{cases}#1 &\to #2\\#3&\mapsto#4\end{cases}}

\newcommand{\then}{\Longrightarrow}

\newcommand{\StepInit}{\newcounter{Step}}

\newcommand{\Step}[1]{\stepcounter{Step}\textbf{Step \theStep} (#1).}

\usepackage{todonotes}
\setuptodonotes{inline}

\author{Thomas Ragel\thanks{CEREMADE, Paris Dauphine University, PSL Research Institute, France.} \ and Bruno Ziliotto\thanks{Toulouse School of Economics, Université Toulouse Capitole, Institut Mathématique de Toulouse, CNRS, France.}}
\title{Constant Payoff Property in Zero-Sum Stochastic Games with a Finite Horizon}

\date{}
\begin{document}
\maketitle
\bibliographystyle{plain}
\begin{abstract}
This paper examines finite zero-sum stochastic games and demonstrates that when the game's duration is sufficiently long, there exists a pair of approximately optimal strategies such that the expected average payoff after any fraction of the game's duration is close to the value. This property, known as the \textit{constant payoff property}, was previously established only for absorbing games and discounted stochastic games.
\end{abstract}
\section*{Introduction}
Zero-sum stochastic games \cite{SH53} model dynamic interactions between two adversarial players. At each stage $m \geq 1$, Player 1 selects an action $i_m$ and Player 2 selects an action $j_m$ simultaneously, possibly using randomization. Player 1 receives a stage payoff $g_m:=g(\omega_m,i_m,j_m)$, and Player 2 receives $-g_m$, where $\omega_m$ is a random variable called the \textit{state}. The distribution of $\omega_m$ depends only on $(\omega_{m-1},i_{m-1},j_{m-1})$. Moreover, at the end of each stage $m$, players are informed of $i_m$, $j_m$ and $\omega_{m+1}$. 

In the \textit{$n$-stage game}, the total payoff is the expectation of the Cesaro mean $\frac{1}{n} \sum_{m=1}^n g_m$, while in the \textit{$\lambda$-discounted game}, the total payoff is the expectation of the Abel mean $\sum_{m \geq 1} \lambda (1-\lambda)^{m-1} g_m$. In both games, Player 1 aims to maximize the total payoff while Player 2 seeks to minimize it. Unless explicitly stated otherwise, the state space and action sets are assumed to be finite (\textit{finite stochastic game}). The value of the $n$-stage game is denoted by $v_n$, and the value of the $\lambda$-discounted game is denoted by $v_\lambda$. Significant effort has been dedicated to studying the asymptotic behavior of stochastic games as $n$ tends to infinity and $\lambda$ goes to 0. 
A seminal result by Bewley and Kohlberg \cite{BK76} shows that $(v_\lambda)$ and $(v_n)$ converge to the same limit $v^*$, called the \textit{limit value}. Recently, Attia and Oliu-Barton \cite{AOB19} provided a characterization of $v^*$. However, the limit value may not exist when the state is unobserved or when the state space is infinite \cite{Z15}. Similarly, it may not exist if one of the action sets is infinite \cite{vigeral13,SV15b,Z13}. 

In this paper, we study the \textit{constant payoff property}, which originated in the paper \cite{SVV10}. That work shows that, in the context of single decision-maker problems, the uniform convergence\footnote{Pointwise convergence alone is not sufficient for this statement to hold, as illustrated by an example in \cite{LS92}.} of $(v_n)$ implies that, for all $t \in \intoo{0}{1}$ and $n$ sufficiently large, when the player uses $\varepsilon$-optimal strategies in the $n$-stage game, the expectation of $\frac{1}{n} \sum_{m=1}^{\ceil{tn}} g_m$ is approximately equal to $t v^*$, up to an $O(\varepsilon)$ term: the average payoff is approximately \textit{constant}. In particular, it can not happen that the player gets far more than the value during the first $n/2$ stages, and then far less during the last $n/2$ stages. A similar property holds for discounted evaluations: for any $M \in \N$ and sufficiently small $\lambda$, if the player uses $\varepsilon$-optimal strategies, the expectation of the cumulated payoff $\sum_{m=1}^M \lambda(1-\lambda)^{m-1}g_m$ is close to $\left[\sum_{m=1}^M \lambda(1-\lambda)^{m-1}\right] v^*$, up to an $O(\varepsilon)$.
However, \cite{SVV10} also provides an example showing that these results
do not generalize to the two-player setting for any pair of $\eps$-optimal strategies. Nevertheless, they propose the following conjecture: in any two-player stochastic game where $(v_n)$ converges uniformly (possibly with infinite state space or action sets), for each $\varepsilon>0$ and $n$ large enough, \textit{there exists} a pair of $\eps$-optimal strategies that satisfies the constant payoff property.

Regarding finite stochastic games, the conjecture holds due to the existence of a uniform value \cite{MN81}, a result we will revisit in the main body of the paper. However, the strategies constructed in \cite{MN81} are notably intricate, as players' choices can depend on the entire past history of states and actions. In the case of absorbing games, \cite{SV20,OB22}\footnote{The paper \cite{SV20} addresses absorbing games with compact action sets, while \cite{OB22} focuses on finite absorbing games and smooth stochastic games. The latter refers to a specific class of stochastic games characterized by strong assumptions on the rate of transitions between states when players employ optimal strategies.} demonstrate that for any $\varepsilon > 0$ and sufficiently large $n$, there exists a pair of $\varepsilon$-optimal Markovian strategies that satisfies the constant payoff property. Unlike the strategies in \cite{MN81}, Markovian strategies depend only on the current state and stage, making them significantly simpler. For general stochastic games with discounted payoffs, a similar result was established in \cite{OBZ18}.

This paper adds to this body of work by demonstrating that in general stochastic games, the aforementioned property holds for $n$-stage payoffs too:
for any $\varepsilon > 0$ and sufficiently large $n$, there exists a pair of $\varepsilon$-optimal Markovian strategies in the $n$-stage game that satisfies the constant payoff property.
The proof uses a class of Markovian strategies, called \textit{adapted strategies}, similar to the asymptotically optimal strategies considered in \cite{Z18g}. At each stage $m$, these strategies play optimally in a discounted game with a discount factor of $\frac{1}{n-m+1}$. This discount factor reflects the importance of stage $m$ relative to the remaining stages of the game. We consider a variation where the discount factor is piecewise constant, updated only at regular intervals. We prove that these adapted strategies are asymptotically optimal in the $n$-stage game and satisfy the constant payoff property. To establish their asymptotic optimality, we employ the same operator-based approach used in \cite{Z18g}. That technique is also reminiscent of an argument by Neyman, that used the operator approach to prove that if $(v_\lambda)$ has bounded variation, then $(v_n)$ converges (see \cite[Theorem C.8, p.177]{sorin02b} and \cite[Theorem 4, p.401]{NS03b}). 
\\
To establish the constant payoff property, we rely on the fact that, by construction, adapted strategies behave locally as optimal strategies in the corresponding discounted game. Using results from \cite{OBZ18}, we deduce that these strategies yield a constant payoff within each block, thereby proving the constant payoff property. The main challenge lies in determining an appropriate block size. Specifically, the block size must be sufficiently large for the results of \cite{OBZ18} to hold. However, if the block size is too large, the strategy may fail to update the discount factor frequently enough, potentially leading to suboptimal outcomes. Our proof technique is largely independent of \cite{SV20, OB22}, which heavily relied on the specific structure of optimal strategies in absorbing games.

The structure of the paper is as follows. In \CRef{sec:mod}, we introduce the stochastic game model and state the main results, \CRef{CP-MainProp,CP-MainTheorem,cor:main}. \CRef{sec:opt} is dedicated to proving \CRef{CP-MainProp}. In \CRef{sec:tech}, we present preliminary results necessary for the proof of Theorem \ref{CP-MainTheorem}, which is completed in \CRef{sec:constant}. Finally, \CRef{sec:perspectives} discusses possible generalizations and future research directions. 
\section{Model and Main Results} \label{sec:mod}
\subsection{Stochastic Games}

Given a finite set $A$, we denote by $\Delta(A)$ the set of probability distributions over $A$. 

\paragraph{Model}
A zero-sum stochastic game is described by a tuple $\Gamma=(\Omega,I,J,g,q)$, where $\Omega$ is the state space, $I$ is Player 1's action set, $J$ is Player 2's action set, $g:\Omega \times I \times J \rightarrow \R$ is the payoff function, and $q: \Omega \times I \times J \rightarrow \Delta(\Omega)$ is the transition function. We assume that $\Omega$, $I$ and $J$ are finite sets. 
\\
The game proceeds as follows: at each stage $m \geq 1$, simultaneously and independently, Player 1 picks $i_m \in I$ and Player 2 picks $j_m \in J$. Player 1 receives the stage payoff $g_m:=g(\omega_m,i_m,j_m)$, while Player 2 receives $-g_m$. The next state $\omega_{m+1}$ is drawn from the distribution $q(\omega_m,i_m,j_m)$. Players observe $(\omega_{m+1},i_m,j_m)$, and the game proceeds to the next stage. 
\paragraph{Strategies}
An element of $H_m:=(\Omega \times I \times J)^{m-1} \times \Omega$ is called an \textit{history before stage $m$}. 
A \textit{strategy} for Player 1 is a collection of mappings $(\sigma_m)_{m \geq 1}$, where $\sigma_m: H_m \rightarrow \Delta(I)$, with the following interpretation: if the history before stage $m$ is $h_m:=(\omega_1,i_1,j_1,\dots,\omega_{m-1},i_{m-1},j_{m-1},\omega_m)$, then Player 1 draws $i_m$ according to the distribution $\sigma(h_m)$. 
Similarly, a strategy for Player 2 is a collection of mappings $(\rho_m)_{m \geq 1}$, where $\rho_m: H_m \rightarrow \Delta(J)$. 
\\
A \textit{Markov strategy} is a strategy that plays according to the current state and the current stage only. A Markov strategy for Player 1 can be identified with a mapping $\sigma :\N \times \Omega \rightarrow \Delta(I)$, and a Markov strategy for Player 2 can be identified with a mapping $\rho :\N \times \Omega \rightarrow \Delta(J)$. 

A \textit{stationary strategy} is a strategy that plays according to the current state only. A stationary strategy for Player 1 can be identified with a mapping $x:\Omega \rightarrow \Delta(I)$, and a stationary strategy for Player 2 can be identified with a mapping $y:\Omega \rightarrow \Delta(J)$. 
\paragraph{$n$-stage game and $\lambda$-discounted game}
An initial state $\omega$ and a pair of strategies $(\sigma,\rho)$ induce a probability measure on the each set of histories $H_m$. Thanks to the Kolmogorov extension theorem, it can be extended uniquely to a probability measure on $(\Omega \times I \times J)^{\N}$, that is denoted by $\P^{\omega}_{\sigma,\rho}$. The expectation with respect to $\P^{\omega}_{\sigma,\rho}$ is denoted by $\E^\omega_{\sigma,\rho}$. 
\\
The \textit{$n$-stage game} $\Gamma_n(\omega)$ is the normal-form game $(\Sigma,T,\gamma^{\omega}_n)$, where $\gamma^{\omega}_n$ is the payoff function defined by
\begin{equation*}
\gamma^{\omega}_n(\sigma,\rho)=\E^{\omega}_{\sigma,\rho} \left(\frac{1}{n} \sum_{m=1}^n g_m \right).
\end{equation*}
This game has a value \cite{SH53}, that is denoted by $v_n(\omega)$:
\begin{equation*}
v_n(\omega):=\max_{\sigma \in \Sigma} \min_{\rho \in T} \gamma^\omega_n(\sigma,\rho)=\min_{\rho \in T}  \max_{\sigma \in \Sigma} \gamma^\omega_n(\sigma,\rho).
\end{equation*}
Let $\varepsilon>0$. A strategy $\sigma \in \Sigma$ of Player 1 is \textit{optimal} (resp., $\varepsilon$-optimal) in $\Gamma_n(\omega)$ if for all $\rho \in T$, $\gamma_n^{\omega}(\sigma,\rho) \geq v_n(\omega)$ (resp., $\gamma_n^{\omega}(\sigma,\rho) \geq v_n(\omega)-\varepsilon$). A strategy $\rho \in T$ of Player 2 is \textit{optimal} (resp., $\varepsilon$-optimal) in $\Gamma_n(\omega)$ if for all $\sigma \in \Sigma$, $\gamma_n^{\omega}(\sigma,\rho) \leq v_n(\omega)$ (resp., $\gamma_n^{\omega}(\sigma,\rho) \leq v_n(\omega)+\varepsilon$). An optimal strategy in $\Gamma_n$ is a strategy that is optimal in $\Gamma_n(\omega)$ for any $\omega$.

Let $\lambda \in (0,1]$. The \textit{$\lambda$-discounted game} is the normal-form game $(\Sigma,T,\gamma^{\omega}_\lambda)$, where
$\gamma^{\omega}_\lambda$ is the payoff function defined by
\begin{equation*}
\gamma^{\omega}_\lambda(\sigma,\rho)=\E^{\omega}_{\sigma,\rho} \left(\sum_{m \geq 1} \lambda(1-\lambda)^{m-1} g_m \right).
\end{equation*}
This game has a value \cite{SH53}, that is denoted by $v_\lambda(\omega)$:
\begin{equation*}
v_\lambda(\omega):=\max_{\sigma \in \Sigma} \min_{\rho \in T} \gamma^\omega_\lambda(\sigma,\rho)=\min_{\rho \in T}  \max_{\sigma \in \Sigma} \gamma^\omega_\lambda(\sigma,\rho).
\end{equation*}
The notion of optimal strategy and $\varepsilon$-optimal strategy can be defined similarly as in the $n$-stage game. 
\\

Let us recall that, using the fact that $\Omega$, $I$ and $J$ are finite, $(v_n)$ and $(v_\lambda)$ both converge (as $n$ tends to $\pinf$ and $\lambda$ tends to $0$) to the same limit $v^*$, called \textit{limit value} \cite{BK76}.

\subsection{Discounted Optimal Profiles, Adaptive Profiles and Main Results}

We now introduce two concepts of strategy families that play a crucial role in the paper. 
\begin{definition}
A family of stationary strategy pairs $(x_\lambda,y_\lambda)_\lambda$ is a \textit{discounted optimal stationary profile} if for all $\lambda \in \intof{0}{1}$, 
$(x_\lambda,y_\lambda)$ is a pair of optimal strategies in $\Gamma_\lambda$. 
\end{definition}

\begin{notation*}
We introduce some piece of notations that will be explained after Definition \ref{def:adapted}. 
	Let $n \geq 1$. Given $a_n \in \N$, let us define, 
	\begin{itemize}
		\item for all $m \in \Intn{1}{n}$, \quad $k(m) := \floor{\f{m-1}{a_n}}$,
		\item $p_n := \floor{\f{n}{a_n}}$,
		\item for all $k \in \Intn{0}{p_n - 1}$, \quad $\lambda_k^n:= 1/(n-k a_n)$.
	\end{itemize}
\end{notation*}
\begin{definition} \label{def:adapted}
	A family of strategy pairs $(\sigma_n,\rho_n)_n$ is an \textit{adapted profile} if there exists a discounted optimal stationary profile $(x_\lambda,y_\lambda)_\lambda$ and a sequence of positive integers $(a_n)_{n \geq 1}$ such that $a_n/n$ tends to $0$ and such that for all $n \geq 1$, the strategy $\sigma_n$ (resp., $\rho_n$) plays $x_{\lambda_{k(m)}^n}$ (resp., $y_{\lambda_{k(m)}^n}$) at each stage $m \in \Intn{1}{n}$.
\end{definition}
\begin{remark}
An adapted profile is a specific case of a family of Markov strategy pairs.
\end{remark}
To understand the rationale behind Definition \ref{def:adapted}, it is helpful to consider the case where $a_n=1$, which corresponds to $p_n=n$, $k(m)=m-1$, and $\lambda^n_{k(m)}=1/(n-m+1)$. A each stage $m$, the strategy pair $(\sigma_n,\rho_n)$ plays optimally in the game with a discount factor $1/(n-m+1)$. This aligns with the type of strategy considered in \cite{Z18g}, where the discount factor used at stage $m$ reflects the relative weight of the stage compared to the duration of the remaining game. Adapted profiles generalize such strategies. 
The set of integers $\left\{1,\dots,n\right\}$ is divided into $p_n$ blocks of size $a_n$, plus one final block of size $n-p_n a_n$. The integer $k(m)$ represents the index of the block to which $m$ belongs, with the first block being $k=0$. 
On the $k$-th block, $(\sigma_n,\rho_n)$ plays optimally in the discounted game with a discount factor $\lambda^n_k=1/(n-k(m)a_n)$. This discount factor represents the relative weight of the first stage of the block compared to the duration of the remaining game. Thus, adapted profiles can be viewed as block-based extensions of the strategies discussed in \cite{Z18g}, where the discount factor is updated only at the beginning of each block rather than at every stage. 
\\
The condition $a_n/n \rightarrow 0$ ensures that the discount factor is updated frequently enough. 
This property is critical for proving that adapted profiles are almost optimal when the duration is long (\textit{asymptotical optimality}), as we shall demonstrate. 

\begin{definition}
A family of strategy pairs $(\sigma_n,\rho_n)_n$ is an \textit{asymptotically optimal profile} if there exists $\varepsilon_n \rightarrow 0$ such that for all $n\in\N$, the strategies $\sigma_n$ and $\rho_n$ are $\varepsilon_n$-optimal in $\Gamma_n$. 
\end{definition}

\begin{definition}
A family of strategy pairs $(\sigma_n,\rho_n)_{n}$ satisfies the \textit{constant payoff property} if for all $t \in \intoo{0}{1}$ and $\omega \in \Omega$,
$$
 \lim_{n \rightarrow+\infty}  \E^\omega_{\sigma_n, \rho_n}\pac{\f{1}{tn} \Sum{m=1}{\ceil{tn}} g_m} = v^*(\omega).
$$
\end{definition}
A family of strategies satisfies the constant payoff property if for $n$ large enough, the expected average payoff after any positive fraction of the $n$-stage game is approximately equal to the limit value (which itself is close to the $n$-stage value). This property is not straightforward, as one might expect that, in a dynamic setting, a player could sacrifice stage payoffs early in the game to improve the state and secure a better payoff later. However, this is not possible when the strategy profile satisfies the constant payoff property.

We are now ready to state our main results. 
\begin{theorem} \label{CP-MainProp}
Any adapted profile is asymptotically optimal.
\end{theorem}

\begin{theorem} \label{CP-MainTheorem}
There exists an adapted profile that satisfies the constant payoff property. 
\end{theorem}
Combining the two theorems gives the following result, which was the main purpose of the paper.
\begin{theorem} \label{cor:main}
There exists a family of Markovian strategy pairs that is asymptotically optimal and that satisfies the constant payoff property.  
\end{theorem}
\subsection{Existence of the uniform value and constant payoff property}
Mertens and Neyman \cite{MN81} proved that any stochastic game has a \textit{uniform value}, which implies that for each $\varepsilon>0$, there exists $n_0 \geq 1$ and a pair of strategies $(\sigma,\rho)$ such that for all $n \geq n_0$, $\sigma$ and $\rho$ are $\varepsilon$-optimal in $\Gamma_n$. 
As a result, for any $t \in ]0,1[$ and $\omega \in \Omega$, 
\begin{equation*}
\liminf_{n \rightarrow +\infty} \E^{\omega}_{\sigma,\rho} \left(\frac{1}{tn} \sum_{m=1}^{\ceil{tn}} g_m \right) \geq v^*(\omega)-\varepsilon \quad \text{and} \quad 
\limsup_{n \rightarrow +\infty} \E^{\omega}_{\sigma,\rho} \left(\frac{1}{tn} \sum_{m=1}^{\ceil{tn}} g_m \right) \leq v^*(\omega)+\varepsilon.
\end{equation*}
By diagonal extraction, we can construct a sequence of strategies $(\sigma_n,\rho_n)$
that is asymptotically optimal and satisfies the constant payoff property. However, in general, these strategies are not Markovian. Therefore, to prove Theorem \ref{cor:main}, we must adopt a different approach to constructing the strategies, one that does not rely on the Mertens and Neyman framework.
\subsection{Proof outline}
\paragraph{Proof of Theorem \ref{CP-MainProp}}
Theorem \ref{CP-MainProp} has been proved in \cite{Z18g}, in the particular case where $a_n=1$ (see the paragraph after Definition \ref{def:adapted}). The proof in the general case follows a similar approach. Consider $(\sigma_n,\rho_n)$ an adapted profile, and $w^n_m$ the payoff guaranteed by the strategy $\sigma_n$ in the game starting from stage $m$. Using Shapley equations, we establish some inequality that connects $\norm{w^n_m-v_{\lambda^n_{k(m)}}}_\infty$ and $\norm{w^n_{m+1}-v_{\lambda^n_{k(m+1)}}}_\infty$, of the following form:
\begin{equation} \label{ineq:form}
\norm{w^n_m-v_{\lambda^n_{k(m)}}}_\infty \leq b_{m,n} \norm{w^n_{m+1}-v_{\lambda^n_{k(m+1)}}}_\infty+c_{m,n},
\end{equation}
where the $b_{m,n}$ are in $[0,1)$, the $(c_{m,n})$ are positive real numbers, and for all $\varepsilon>0$, $\sum_{m=1}^{\lfloor (1-\varepsilon) n \rfloor} c_{m,n}$ tends to 0 as $n$ tends to infinity. The latter property uses crucially that $a_n/n$ tends to 0. Iterating inequality \ref{ineq:form} from $m=1$ to $\lfloor (1-\varepsilon) n \rfloor$ yields that 
$\norm{w^n_1-v_{\lambda^n_{k(1)}}}_\infty=\norm{w^n_1-v_{\frac{1}{n}}}_\infty$ tends to 0. Because $v_{\frac{1}{n}}$ tends to $v^*$ as $n$ tends to infinity, this proves the theorem.

\paragraph{Proof of Theorem \ref{CP-MainTheorem}}
Let $(\sigma_n, \rho_n)_n$ be a family of Markovian strategy pairs. 
 First, we prove that if $(\sigma_n, \rho_n)_n$ is asymptotically optimal and satisfies the condition that for all $t \in ]0,1[$ and $\omega \in \Omega$, 
\begin{equation} \label{eq:constant}
\E_{\sigma_n, \rho_n}^\omega\pa{v^*(\omega_{\ceil{tn}+1}) - v^*(\omega)} \ttu{n \to \pinf} 0,
\end{equation}
then $(\sigma_n, \rho_n)_n$ satisfies the constant payoff property. In particular, this holds when $(\sigma_n,\rho_n)$ is an adapted profile. A similar result was established in the discounted case \cite{OBZ18} (see Proposition 3.1). 

Next, we use a result from \cite{OBZ18}, that implies that the difference between the expectations of $v^*(\omega_{\ell})$ and $v^*(\omega_m)$ for two stages $\ell$ and $m$ within the same block $k$ is bounded by an error term $\beta_n$, which vanishes as $n$ tends to infinity. This stems from the fact that within a block, adapted strategies play optimally in a discounted game with a fixed discount factor. Summing over all blocks, we obtain that for any $m$ not too close to $n$, the expectation of $v^*(\omega_m)$ is close to $v^*(\omega)$, with an error term of order $p_n \beta_n$. By selecting $p_n$ such that $p_n \beta_n$ tends to 0 as $n$ tends to infinity, we obtain \Cref{eq:constant}, and Theorem \ref{CP-MainTheorem} follows. 

Interestingly, to ensure that $p_n \beta_n$ tends to 0, it is essential that the number of blocks $p_n$ is not too large. On the other hand, to prove that adapted strategies are asymptotically optimal, we need the fact that $a_n/n$ tends to 0, meaning that $p_n$ must be sufficiently large. Thus, there is a trade-off between ensuring the constant payoff property and asymptotic optimality. Consequently, the sequence $(a_n)$ must be chosen carefully to satisfy both conditions simultaneously, which constitutes the main difficulty of the proof. 

%

\section{Proof of \CRef{CP-MainProp}} \label{sec:opt}

\begin{proof}
Let $(\sigma_n,\rho_n)_n$ be an adapted profile, let $(x_\lambda,y_\lambda)_\lambda$ be the corresponding discounted optimal profile and $(a_n)_n$ be the corresponding sequence. Let us recall the following definitions:
 $$k(m) := \floor{\f{m-1}{a_n}}, \quad p_n:=\lfloor n/a_n \rfloor \qandq \lambda_k^n := 1/(n-k a_n).$$ 
Let us also define $\mu^n_m := \lambda^n_{k(m)}$. 
Hence, at any stage $m \in \Intn{1}{n}$, the strategies $\sigma_n$ and $\rho_n$ play respectively $x_{\mu_m^n}$ and $y_{\mu_m^n}$. 
\\
Define $w^n_m$ as the payoff guaranteed by $\sigma_n$ in the $n$-stage game starting from stage $m$. Let $\omega \in \Omega$. Define
\begin{equation*}
g_j:=\sum_{i \in I} x_{\mu^n_m}(i|\omega) g(\omega,i,j),
\end{equation*}
where $x_{\mu^n_m}(i|\omega)$ designates the probability that action $i$ is chosen at state $\omega$, under the stationary strategy $x_{\mu^n_m}$. 
The quantity $g_j$ represents the expected stage payoff at state $\omega$, given that Player 1 plays the mixed action $x_{\mu^n_m}(.|\omega)$ and Player 2 plays $j$. 
Moreover, define
\begin{equation*}
w_j:=\sum_{(i,\omega') \in I \times \Omega} x_{\mu^n_m}(i|\omega) q(\omega'|\omega,i,j) w_{m+1}^n(\omega') \quad \text{and} \quad
w'_j:=\sum_{(i,\omega') \in I \times \Omega} x_{\mu^n_m}(i|\omega) q(\omega'|\omega,i,j) v_{\mu^n_m}(\omega').
\end{equation*}
The quantity $w_j$ is the expectation of $w_{m+1}^n(\tilde{\omega})$, where $\tilde{\omega}$ is the (random) next state, given that Player 1 plays the mixed action $x_{\mu^n_m}(.|\omega)$ and Player 2 plays $j$. The quantity $w'_j$ can be understood similarly, up to replacing $w_{m+1}^n(\omega')$ by $v_{\mu^n_m}(\omega')$.
Let $\omega \in \Omega$. 
Using Shapley's equations \cite{SH53}, we have
\begin{equation*}
w^n_m(\omega)=\min_{j \in J} \left\{ \frac{1}{n-m+1} \cdot g_j+\pa{1-\frac{1}{n-m+1}} \cdot w_j \right\}.
\end{equation*}
Similarly, we have
\begin{equation*}
v_{\mu^n_m}(\omega) = \min_{j \in J} \left\{ \mu^n_m \cdot g_j+(1-\mu^n_m) \cdot w'_j \right\}.
\end{equation*}
Consider two finite sets of real numbers $A = \{a_j\}_{j \in J}$ and $B = \{b_j\}_{j\in J}$. The following inequality holds: $$\abs{\min A - \min B} \leq \max_{j \in J} \abs{a_j - b_j}.$$
We deduce that
\begin{eqnarray*}
|w^n_m(\omega)-v_{\mu^n_m}(\omega)| &\leq& 
 |g_j| \abs{\frac{1}{n-m+1}-\mu^n_m}+\left|\pa{1-\frac{1}{n-m+1}} \cdot w_j- (1-\mu^n_m) \cdot w'_j  \right|
 \\
&\leq&
 |g_j| \abs{\frac{1}{n-m+1}-\mu^n_m}+\pa{1-\frac{1}{n-m+1}} \cdot \left|w_j-w'_j\right|+|w'_j|\abs{\frac{1}{n-m+1}-\mu^n_m}
 \\
 &\leq&
2 \norm{g}_\infty \abs{\frac{1}{n-m+1}-\mu^n_m}+ \pa{1-\frac{1}{n-m+1}}\norm{w^n_{m+1}-v_{\mu^n_m}}_\infty
\\
 &\leq&
2 \norm{g}_\infty \abs{\frac{1}{n-m+1}-\mu^n_m}
\\
&+& \pa{1-\frac{1}{n-m+1}}
\left(\norm{w^n_{m+1}-v_{\mu^n_{m+1}}}_\infty+
\norm{v_{\mu^n_{m+1}}-v_{\mu^n_{m}}}_\infty \right),
\end{eqnarray*}
where in the second-to-last inequality, we used the fact that $|g_j| \leq \norm{g}_\infty$ and $\left|w_j-w'_j\right| \leq \norm{w^n_{m+1}-v_{\mu^n_m}}_\infty$. 
Multiplying both sides by $\f{n-m+1}{n}$ yields
\begin{align*}
&\frac{n-m+1}{n} \norm{w^n_m-v_{\mu^n_m}}_\infty - \frac{n-m}{n} \norm{w^n_{m+1}-v_{\mu^n_{m+1}}}_\infty 
\\ &\qquad \leq  \f{n-m+1}{n} \pa{2 \norm{g}_\infty \abs{\frac{1}{n-m+1}-\mu^n_m}
+\norm{v_{\mu^n_{m+1}}-v_{\mu^n_m}}_\infty}.
\end{align*}
We deduce that
\begin{equation} \label{eq:tel}
\frac{n-m+1}{n} \norm{w^n_m-v_{\mu^n_m}}_\infty - \frac{n-m}{n} \norm{w^n_{m+1}-v_{\mu^n_{m+1}}}_\infty 
\leq  2 \norm{g}_\infty \abs{\frac{1}{n-m+1}-\mu^n_m}
+\norm{v_{\mu^n_{m+1}}-v_{\mu^n_m}}_\infty.
\end{equation}
Fix $\varepsilon \in (0,1)$, and define $k_0:=\lfloor (1-\varepsilon) p_n \rfloor$. Let 
$u_m:=\frac{n-m+1}{n} \norm{w^n_m-v_{\mu^n_m}}_\infty, m \geq 1$. We have
\begin{eqnarray*}
\sum_{m=1}^{k_0 n} u_m-u_{m+1}=u_1-u_{k_0a_n+1}&=&
\norm{w^n_1-v_{\mu^n_1}}_\infty - \frac{n-k_0 a_n}{n} \norm{w^n_{k_0 a_n+1}-v_{\mu^n_{k_0 a_n+1}}}_\infty
\\
&\geq& \norm{w^n_1-v_{\lambda^n_0}}_\infty - 2\left(\varepsilon+\frac{2 a_n}{n} \right) \norm{g}_\infty,
\end{eqnarray*}
where we used in the last line that by definition, $\mu^n_1=\lambda^n_{k(1)}=\lambda^n_0$, and moreover,  
\\
$n-k_0 a_n \leq n-[(1-\varepsilon)[n/a_n-1]-1] a_n \leq 
\varepsilon n +2a_n.$
Summing  \CRef{eq:tel} from $m = 1$ to $k_0 a_n$, we obtain
\begin{equation} \label{eq:gar}
\norm{w^n_1-v_{\lambda^n_0}}_\infty \leq 2 \left(\varepsilon+\frac{2 a_n}{n} \right) \norm{g}_\infty + 2 \norm{g}_\infty \sum_{m=1}^{k_0 a_n} \left|\frac{1}{n-m+1}-\mu^n_m \right|+ \sum_{m=1}^{k_0 a_n} \norm{v_{\mu^n_{m+1}}-v_{\mu^n_m}}_\infty.
\end{equation}

Let us bound from above the two sums, starting with the first one. For $k  \in \Intn{0}{k_0-1}$ and ${m \in \Intn{k a_n +1}{ (k+1) a_n }}$, we have
\begin{align*}
\abs{\frac{1}{n-m+1}-\mu^n_m} &= \abs{\frac{1}{n-m+1}-\lambda^n_k} \\
&\leq \frac{1}{n-(k+1) a_n}-\frac{1}{n-k a_n}.
\end{align*}
We deduce that
\begin{eqnarray*}
\sum_{m= 1}^{k_0 a_n} \abs{\frac{1}{n-m+1}-\mu^n_m}
&\leq& a_n \sum_{k=0}^{k_0-1} \pa{\frac{1}{n-(k+1) a_n}-\frac{1}{n-k a_n}}
\\
&\leq& a_n \frac{1}{n-k_0 a_n}
\\
&\leq& \frac{a_n}{\varepsilon n}.
\end{eqnarray*}
Because $a_n/n \rightarrow 0$, we deduce that the first sum in \CRef{eq:gar} vanishes as $n$ tends to infinity.

Let us now bound the second sum in \CRef{eq:gar}. The family $(v_\lambda)_\lambda$ has bounded variation, the sequence $(\mu^n_m)_{m \geq 1}$ is increasing, and ${\mu^n_{k_0 a_n+1}=\frac{1}{n-k_0 a_n} \leq \frac{1}{\varepsilon n}}$. We deduce that the second sum vanishes as $n$ tends to infinity. 

Using \CRef{eq:gar}, we deduce that
\begin{eqnarray*}
\limsup_{n \rightarrow +\infty} \norm{w^n_1-v_{\lambda^n_0}}_\infty \leq 2 \varepsilon \norm{g}_\infty.
\end{eqnarray*}
Because $v_{\lambda^n_0} = v_{1/n}$ goes to $v^*$, and $\varepsilon$ is arbitrary, we deduce that 
$\lim_{n \rightarrow +\infty} w^n_1=v^*$, hence the family $(\sigma_n)_n$ is asymptotically optimal. Exchanging the roles of the players, we obtain the proof of \CRef{CP-MainProp}. 
\end{proof}

\section{Preliminary Results for the Proof of Theorem \ref{CP-MainTheorem}} \label{sec:tech}
This section presents two technical results, primarily adapted from \cite{OBZ18}, which will be useful for the proof of \CRef{CP-MainTheorem}. First, we recall the results from \cite{OBZ18} that we will exploit. 

\subsection{Results from the Paper \cite{OBZ18}}
For all $\lambda \in \intoo{0}{1}$ and $t \in [0,1[$, define 
$$\phi(\lambda,t):=\inf \left\{M \geq 1, \sum_{m=1}^M \lambda(1-\lambda)^{m-1} \geq t \right\}=\displaystyle \ceil{\frac{\ln(1-t)}{\ln(1-\lambda)}} \in \N \cup \{\pinf\}.$$
Let us state Proposition 
4.4 from \cite{OBZ18}:
\begin{proposition} \label{CP-PropBruno}
	Let $(x_\lambda, y_\lambda)_\lambda$ be a discounted optimal profile. 
	The following properties are equivalent:
	\begin{enumerate}
		\item The family $(x_\lambda, y_\lambda)_\lambda$ satisfies the \textit{discounted constant payoff property} : for all $t \in \intoo{0}{1}$ and $\omega \in \Omega$,
$$
\lim_{\lambda \to 0}  \E^\omega_{x_\lambda, y_\lambda}\pac{ \Sum{m=1}{\phi(\lambda, t)} \lambda(1-\lambda)^{m-1}g_m} = tv^*(\omega).
$$
		
		\item For all $\omega \in \Omega$ and $t \in \intoo{0}{1}$, 
		\quad $\Lim_{\lambda\to 0} \pa{\E^\omega_{x_\lambda, y_\lambda}\pac{v^*(\omega_{\phi(\lambda, t)})} - v^*(\omega)} = 0$.
	\end{enumerate}
\end{proposition}

Moreover, in \cite{OBZ18}, it has been shown that any discounted optimal profile verifies the discounted constant payoff property:

\begin{theorem} \label{CP-ThmBruno}
	Any discounted optimal profile $(x_\lambda, y_\lambda)_\lambda$ verifies the discounted constant property. In particular, for all $\omega \in \Omega$ and $t \in \intoo{0}{1}$, 
	$$\Lim_{\lambda\to 0} \pa{\E^\omega_{x_\lambda, y_\lambda}\pac{v^*(\omega_{\phi(\lambda, t)})} - v^*(\omega)} = 0.$$
\end{theorem}

\subsection{Uniform Convergence in Proposition \ref{CP-PropBruno}}
The convergence in Theorem \ref{CP-ThmBruno} is pointwise in $t$. The following result shows that the convergence is uniform on any interval $]0,T]$, $T \in ]0,1[$.
\begin{proposition} \label{CP-propunif}
For all $\omega \in \Omega$ and $T \in \intoo{0}{1}$, we have
\begin{align*} \label{CP-firstgoal}
\sup_{t \in \intof{0}{T}} \E_{x_{\lambda}, y_{\lambda}}^\omega\pac{v^*(\omega_{\phi({\lambda},t)}) - v^*(\omega)} \ttu{\lambda \to 0} 0.
\end{align*}
\end{proposition}
We first need a lemma that is an elementary variation of the Arzela-Ascoli theorem. 
\begin{lemma} \label{CP-lemmauniform}
	Let $(f_{\lambda})_{\lambda\in \intoo{0}{1}}$ be a family of functions such that:
	\begin{itemize}
		\item For all $\lambda \in \intoo{0}{1}$, $f_{\lambda}$ is a function $f_{\lambda} : \intoo{0}{1} \to \R$.
		\item $f_{\lambda} \ttu{\lambda \to 0} 0$ pointwise.
		\item There exists $K > 0$ such that, for all $\lambda>0$ small enough, and all $t,t' \in \intoo{0}{1}$ with $t < t'$:
		
		\begin{equation}
			\abs{f_{\lambda}(t') - f_{\lambda}(t)} \leq K\pa{(t'-t) + \lambda}.
		\end{equation}
	\end{itemize}
	Then $f_{\lambda} \ttu{\lambda \to 0} 0$ uniformly.
\end{lemma}
\begin{proof}
	Let $\eps \in ]0,1/2]$ and consider the finite set $S = \Set{k\eps}{k \in\Intn{1}{\floor{\f{1}{\eps}}-1}}$.

	The limit $f_{\lambda} \ttu{\lambda \to 0} 0$ is uniform on the finite set $S$. Let us show that it is uniform on $\intoo{0}{1}$. 

	Let $\Lambda < \eps$ such that $\forall \lambda \in \intoo{0}{\Lambda}$, the two following inequalities are satisfied:
	
\begin{equation} \label{CP-lipschitz}
	\forall 0<t<t'<1, \quad \abs{f_{\lambda}(t') - f_{\lambda}(t)} \leq K\pa{(t'-t) + \lambda},
		\end{equation}
\begin{equation}
\label{CP-finiteunif}
\forall s \in S, \quad \abs{f_{\lambda}(s)} \leq \eps.
\end{equation}

Let $t \in \intoo{0}{1}$ and an integer $k$ 
such that $k \varepsilon \in S$ and $\abs{t - k\eps} \leq 2\eps$. Combining \CRef{CP-lipschitz} and \CRef{CP-finiteunif}, one has, for all $\lambda < \Lambda$:

\begin{align*}
	\abs{f_{\lambda}(t)} &\leq \abs{f_{\lambda}(t)- f_{\lambda}(k\eps)} + \abs{f_{\lambda}(k\eps)} \\
	&\leq \pa{(t-k\eps) + \lambda} K + \eps \\
	&\leq \pa{3K + 1}\eps,
\end{align*}
which proves that $\norm{f_{\lambda}}_\infty \ttu{\lambda \to 0} 0$.
\end{proof}

\begin{proof}[Proof of \CRef{CP-propunif}]
Using \CRef{CP-ThmBruno}, 
 for all $\omega \in \Omega$ and $t \in \intoo{0}{1}$:
\begin{align} \label{CP-simplelim}
	\Lim_{\lambda \to 0} \E_{x_\lambda, y_\lambda}^\omega\pac{v_\lambda(\omega_{\phi(\lambda,t)}) - v_\lambda(\omega)} = 0.
\end{align}
Let us use \CRef{CP-lemmauniform} to show that this limit is uniform on $\intof{0}{T}$ for all $T \in \intoo{0}{1}$.
Fix $\omega \in \Omega$ and let us define, for all $\lambda \in \intof{0}{1}$, the function 
	$${f_\lambda : t \mapsto (1-\lambda)^{\phi(\lambda, t)-1}\E^\omega_{x_\lambda,y_\lambda}\pac{v_\lambda\pa{\omega_{\phi(\lambda, t)}} - v_\lambda(\omega)}},$$
	defined on $\intoo{0}{1}$. 
\CRef{CP-simplelim} implies that $f_\lambda(t) \ttu{\lambda \to 0} 0$ for all $t \in \intoo{0}{1}$. 
\\
Shapley's equations \cite{SH53} yield, for any $p \geq 1$:
	
	\begin{align} \label{CP-ShapEq}\tag{E\textsubscript{p}}
		v_\lambda(\omega) = \E^\omega_{x_\lambda, y_\lambda}\pac{\Sum{m=1}{p-1} \lambda(1-\lambda)^{m-1}g_m} + (1-\lambda)^{p-1} \E^\omega_{x_\lambda,y_\lambda}\pac{v_\lambda(\omega_p)}.
	\end{align}
Let $a$ and $b$ be two integers such that $1 \leq a<b$. Making the difference of \CRef{CP-ShapEq} with $p = b$ and \CRef{CP-ShapEq} with $p = a$ yields:
	
\begin{align} \label{CP-ShapEqab}
\begin{split} 
	&\abs{(1-\lambda)^{b-1}\E^\omega_{x_\lambda, y_\lambda}\pac{v_\lambda(\omega_{b})-v_\lambda(\omega)} 
	- (1-\lambda)^{a-1}\E^\omega_{x_\lambda, y_\lambda}\pac{v_\lambda(\omega_{a})-v_\lambda(\omega)}} \\
	&= \abs{-\E^\omega_{x_\lambda, y_\lambda}\pac{\Sum{m=a}{b-1} \lambda(1-\lambda)^{m-1} g_m} + \pa{(1-\lambda)^{a-1}  - (1-\lambda)^{b-1}} v_\lambda(\omega)} \\
	&\leq 2\norm{g}_\infty\pa{(1-\lambda)^{a-1}  - (1-\lambda)^{b-1}}.
\end{split}
\end{align}
\\
Let us consider $t,t' \in \intoo{0}{1}$ with $t < t'$. We recall that $\phi(\lambda, t) = \ceil{\f{\ln(1-t)}{\ln(1-\lambda)}}$. Using \CRef{CP-ShapEqab}:
	
\begin{align*}
\begin{split}
	\abs{f_\lambda(t') - f_\lambda(t)} 
	&\leq \f{2}{1-\lambda}\norm{g}_\infty \pa{(1-\lambda)^{\phi(\lambda,t)}  - (1-\lambda)^{\phi(\lambda,t')}} \\
	&\leq \pa{(1-\lambda)^{\f{\ln(1-t)}{\ln(1-\lambda)}}  - (1-\lambda)^{\f{\ln(1-t')}{\ln(1-\lambda)}+1}} \\
	&= \f{2}{1-\lambda}\norm{g}_\infty  \pac{(1-t) - (1-t')(1-\lambda)}\\
	&= \f{2}{1-\lambda}\norm{g}_\infty \pac{(t'-t) + \lambda(1-t')} \\
	&\leq \f{2}{1-\lambda} \norm{g}_\infty \pa{(t'-t) + \lambda}.
\end{split}
\end{align*}
For $\lambda \leq \f{1}{2}$, we hence have:
\begin{equation*}
\abs{f_\lambda(t') - f_\lambda(t)} \leq 4\norm{g}_\infty \pa{(t'-t) + \lambda} .
\end{equation*}
Using \CRef{CP-lemmauniform}, for all $\omega \in \Omega$, $f_{\lambda}$ converges uniformly on $\intoo{0}{1}$. It follows that for all $T \in \intoo{0}{1}$,  

\begin{equation*}
	\Sup_{t \in \intof{0}{T}} \E_{x_\lambda,y_\lambda}^\omega\pac{v_\lambda(\omega_{\phi(\lambda,t)}) - v_\lambda(\omega)} \ttu{\lambda \to 0} 0.
\end{equation*}
As $v_\lambda \ttu{\lambda \to 0} v^*$ uniformly, \CRef{CP-propunif} is proved.
\end{proof}

\subsection{A Sufficient Condition for the Constant Payoff Property}

We show an analogous result to the implication (\emph{2}. $\then$ \emph{1}.) in \CRef{CP-PropBruno} in the finite-horizon framework, as stated in the following proposition:

\begin{proposition} \label{CP-PropBrunoAsympt}
	Let $(\sigma_n, \rho_n)_n$ be a family of Markov strategies that is an asymptotically optimal profile. If for all $\omega \in \Omega$ and $t \in \intoo{0}{1}$, $$\E_{\sigma_n, \rho_n}^\omega\pa{v^*(\omega_{\ceil{tn}+1}) - v^*(\omega)} \ttu{n \to \pinf} 0,$$ then $(\sigma_n, \rho_n)_n$ satisfies the constant payoff property. 
\end{proposition}
\begin{remark}
Similar to the discounted case, we believe that the sufficient condition for the constant payoff property outlined in Proposition \ref{CP-PropBrunoAsympt} is, in fact, also necessary. However, we do not provide a formal proof of this assertion, as it is not required for establishing Theorem \ref{CP-MainTheorem}.
\end{remark}
%

\begin{proof}
By assumption, there exists a vanishing sequence of nonnegative numbers $(\eps_n)$ such that for all $n$, $\sigma_n$ and $\rho_n$ are $\varepsilon_n$-optimal in $\Gamma_n$. 
	Let $\omega \in \Omega$ and $n \geq 1$. Define $\delta_n:= \E^\omega_{\sigma_n, \rho_n}\pac{\f{1}{n} \Sum{m=1}{n} g_m} - v_n(\omega)$. Because the strategies $\sigma_n$ and $\rho_n$ are $\varepsilon_n$-optimal, we have $|\delta_n| \leq \eps_n \ttn 0$.

	Now, take $t \in \intoo{0}{1}$. 
	One has:
	
	\begin{align}
		\E^\omega_{\sigma_n, \rho_n}\pac{\f{1}{n} \Sum{m=1}{n} g_m} 
		&= \E^\omega_{\sigma_n, \rho_n}\pac{\f{1}{n} \Sum{m=1}{\ceil{tn}} g_m}  + 
		\E^\omega_{\sigma_n, \rho_n}\pac{\f{1}{n} \Sum{m=\ceil{tn}+1}{n} g_m} \nonumber 
		\\
		&= \E^\omega_{\sigma_n, \rho_n}\pac{\f{1}{n} \Sum{m=1}{\ceil{tn}} g_m} + 
		\f{n-\ceil{tn}}{n} \E^\omega_{\sigma_n, \rho_n}\pac{\f{1}{n-\ceil{tn}} \Sum{m=1}{n-\ceil{tn}} g_{m+\ceil{tn}}} \label{eq:gamman}.
	\end{align}

	As a result,
	
	\begin{align*} 
	\begin{split}
		\delta_n &= \pa{\E^\omega_{\sigma_n, \rho_n}\pac{\f{1}{n}\Sum{m=1}{\ceil{tn}} g_m} - tv_n(\omega)} \\
		&\quad + \f{n-\ceil{tn}}{n} \pa{\E^\omega_{\sigma_n, \rho_n}\pac{\f{1}{n-\ceil{tn}} \Sum{m=1}{n-\ceil{tn}} g_{m+\ceil{tn}}} - \E^\omega_{\sigma_n, \rho_n}\pac{v_{n-\ceil{tn}}(\omega_{\ceil{tn}+1})}} \\
		&\quad + \f{n-\ceil{tn}}{n} \pa{\E^\omega_{\sigma_n, \rho_n}\pac{v_{n-\ceil{tn}}(\omega_{\ceil{tn}+1})} - v_n(\omega)} \\
		&\quad + \pa{\f{n-\ceil{tn}}{n} - (1-t)}v_n(\omega).
	\end{split}
	\end{align*}
	Clearly, 
	
	$$
		\pa{\f{n-\ceil{tn}}{n} - (1-t)}v_n(\omega) \ttu{n \to \pinf} 0,
	$$
	and, using our hypothesis, 
	
	$$
		\E_{\sigma_n, \rho_n}^\omega\pa{v_{n-\ceil{tn}}(\omega_{\ceil{tn}+1}) - v_{n}(\omega)} \ttu{n \to \pinf} 0.
	$$
	To prove the proposition, it is thus enough to show that 
	\begin{equation} \label{CP-sufcond-eq:last}
	 \E^\omega_{\sigma_n, \rho_n}\pac{\f{1}{n-\ceil{tn}} \Sum{m=1}{n-\ceil{tn}} g_{m+\ceil{tn}}} - \E^\omega_{\sigma_n, \rho_n}\pac{v_{n-\ceil{tn}} (\omega_{\ceil{tn}+1})}  \ttu{n \to \pinf} 0.
	 \end{equation}
	 If this was not the case, there should be some $n$ such that either
	 \begin{equation*}
	 \frac{n-\ceil{tn}}{n} \left[ \E^\omega_{\sigma_n, \rho_n}\pac{\f{1}{n-\ceil{tn}} \Sum{m=1}{n-\ceil{tn}} g_{m+\ceil{tn}}} - \E^\omega_{\sigma_n, \rho_n}\pac{v_{n-\ceil{tn}} (\omega_{\ceil{tn}+1})} \right]
	  < -2\varepsilon_n,
	 \end{equation*}
	 or the reverse inequality with $\varepsilon_n$ instead of $-\varepsilon_n$. Assume the first inequality holds. Combining with \CRef{eq:gamman}, we get
	\begin{equation*}
	\gamma^{\omega}_n(\sigma_n,\rho_n)<
	\E^\omega_{\sigma_n, \rho_n}\pac{\f{1}{n} \Sum{m=1}{\ceil{tn}} g_m} + 
		\f{n-\ceil{tn}}{n} \E^\omega_{\sigma_n, \rho_n}\pac{v_{n-\ceil{tn}} (\omega_{\ceil{tn}+1})}+2\varepsilon_n.
	\end{equation*}

	 By playing $\sigma_n$ until stage $\ceil{tn}$, and then an optimal strategy in $\Gamma_{n-\ceil{tn}}$, Player 1 would thus get strictly more than
	 $\gamma^{\omega}_n(\sigma_n,\rho_n)+2 \varepsilon_n$. This contradicts the fact that $\sigma_n$ and $\rho_n$ are $\varepsilon_n$-optimal, which concludes the proof of the proposition. 
\end{proof}

\section{Proof of \CRef{CP-MainTheorem}} \label{sec:constant}
Let us recall \CRef{CP-MainTheorem}.
\begin{theorem*} 
There exists an adapted profile that satisfies the constant payoff property. 
\end{theorem*}
\begin{proof} ~


\StepInit
\Step{Definition of the strategy} Let $(x_\lambda, y_\lambda)_\lambda$ be a discounted optimal profile.

Using \CRef{CP-propunif}, for all $p \geq 1$, there exists $\mu \in \intoo{0}{1}$ such that:
\begin{equation} \label{eq:psquare}
	\forall \lambda \in \intof{0}{\mu}, \: \forall t \in \intof{0}{7/8}, \: \forall \omega \in \Omega, \quad 
	\abs{\E^{\omega}_{x_\lambda, y_\lambda}\pa{v^*(\omega_{\phi(\lambda,t)}) - v^*(\omega)}} \leq p^{-2}. 
\end{equation}
One can define a sequence $(\mu_p)_p$ in $\intof{0}{1/2}$ such that $\mu_p \ttu{p \to \pinf} 0$, and, for all $p \geq 1$, $\mu_p$ verifies \CRef{eq:psquare}.  
There exists $n_0 \in \N$ such that, for all $n \geq n_0$, the set $\Set{a \in \Intn{2}{n}}{1/a \leq \mu_{\lfloor n/a \rfloor}}$ is non-empty and thus has a minimum $a_n$. By definition, for all $n \geq n_0$,
\begin{equation} \label{eq:mupn}
	2 \leq a_n \leq n \qandq 1/a_n \leq \mu_{\lfloor n/a_n \rfloor}.
\end{equation}
To respect the definition of an adapted profile, one can complete the sequence by defining arbitrarily ${a_1, \dots, a_{n_0-1} \in \N}$. Since the property we are interested in is asymptotic, the exact values of $a_1, \dots, a_{n_0-1} \in \N$ are not relevant.

For each $\varepsilon>0$ and $n$ large enough, we have $1/(\varepsilon n) \leq \mu_{\lfloor n/(\varepsilon n) \rfloor}$, hence $a_n \leq \varepsilon n$. This shows that $a_n/n$ tends to $0$.

Let us define $(\sigma_n,\rho_n)_n$ as the adapted profile corresponding to the discounted optimal profile $(x_\lambda, y_\lambda)_\lambda$ and the sequence of positive integers $(a_n)_n$. More precisely, the strategy $\sigma_n$ (resp., $\rho_n$) plays $x_{\lambda_{k(m)}^n}$ (resp., $y_{\lambda_{k(m)}^n}$) at each stage $m \in \Intn{1}{n}$.

Let us show that it satisfies the constant payoff property. \\


\Step{A bound on the variation of the expected value within a block}
\begin{notation*}
	We define the following notations:
		$$
			\forall \omega \in \Omega, \quad \E^\omega_{n} := \E^\omega_{\sigma_n,\rho_n} \qandq \E^\omega_{\lambda} := \E^\omega_{x_{\lambda}, y_{\lambda}}.
		$$	
\end{notation*}

Let $n \geq n_0$, $k \in \Intn{0}{p_n - 1}$, $j \in \Intn{1}{a_n+1}$, and $\omega \in \Omega$. Let us show that
\begin{equation} \label{eq:var}
| \m{E}^\omega_{n} \left[v^*(\omega_{ka_n+j}) - v^*(\omega_{k a_n+1}) \right]| \leq p_n^{-2}.
\end{equation}
Since
\begin{eqnarray*}
\E^\omega_{n} \left[v^*(\omega_{k a_n+j}) - v^*(\omega_{k a_n+1})\right]
&=&  
\m{E}^{\omega}_{n} \left[\E^{\omega_{k a_n+1}}_{\lambda^n_k}\pa{v^*(\omega_{j}) - v^*(\omega_{1})} \right],
\end{eqnarray*}
it is sufficient to show that for all $\omega' \in \Omega$, all $j' \in \Intn{1}{a_n+1}$, 
\begin{equation} \label{eq:var2}
	\abs{\m{E}^{\omega'}_{\lambda^n_k} \pac{v^*(\omega_{j'})} - v^*(\omega')} \leq p_n^{-2}.
\end{equation}
Let $j' \in \Intn{1}{a_n+1}$ and $\omega' \in \Omega$. The function $\fun{\intoo{0}{1}}{\N \cup \{0\}}{t}{\phi(\lambda^n_k, t)}$ is surjective.  Moreover, by definition of $\lambda^n_k$, we have 
\begin{equation} \label{lambdaineq1}
	\lambda^n_k \leq \f{1}{a_n}.
\end{equation}
Using \CRef{lambdaineq1}, one has:
\begin{align*}
	\Sum{m=1}{j'} \lambda^n_k(1-\lambda^n_k)^{m-1}&= 1-\left(1-\lambda^n_k \right)^{j'}
	\leq 1-\left(1-\frac{1}{a_n}\right)^{a_n+1}.
\end{align*}
As the function $\lambda \mapsto 1 - \pa{1-\lambda}^{\f{1}{\lambda}+1}$ is increasing on $\intof{0}{1/2}$ and $a_n \geq 2$, one has:
\begin{align*}
	\Sum{m=1}{j'} \lambda^n_k(1-\lambda^n_k)^{m-1} \leq \f{7}{8}.
\end{align*}
As a result, there exists $t \in \intof{0}{7/8}$ such that $\phi(\lambda^n_k,t) = j'$.
\CRef{eq:mupn} and \CRef{lambdaineq1} yield $\lambda^n_k \leq \mu_{p_n}$, and, combined with \CRef{eq:psquare}, we obtain
\begin{eqnarray*}
	\abs{\m{E}^{\omega'}_{\lambda^n_k} \left[v^*(\omega_{j'})\right] - v^*(\omega')} =
	\abs{\m{E}^{\omega'}_{\lambda^n_k} \left[v^*(\omega_{\phi(\lambda^n_k,t)})\right] - v^*(\omega')}
\leq p_n^{-2}.
\end{eqnarray*}
We deduce that \CRef{eq:var2} holds, hence \CRef{eq:var} holds. 
\\

\Step{A general bound on the variation of the expected value}
\\
Let $n \geq n_0$, $m \in \Intn{1}{p_n \cdot a_n}$ and $\omega \in \Omega$. Let us show that
\begin{equation} \label{eq:omegam}
	\abs{\m{E}^\omega_{n}\left[v^*(\omega_{m})\right] - v^*(\omega)}
\leq 2p_n^{-1}.
\end{equation}
Indeed, let $k' \in \Intn{0}{p_n-1}$ such that $k' a_n+1 \leq m \leq (k'+1) a_n$. We have
\begin{eqnarray*}
	\abs{\m{E}^\omega_{n}\left[v^*(\omega_{m})\right] - v^*(\omega)} 
	&\leq&
	\sum_{\ell=0}^{k'-1} \abs{\m{E}^\omega_{n}\left[v^*(\omega_{(\ell+1) a_n+1})\right] - v^*(\omega_{\ell a_n+1})} 
	+ \abs{\m{E}^\omega_{n}\left[v^*(\omega_{m})\right] - v^*(\omega_{k' a_n+1})} 
\end{eqnarray*}
Applying \CRef{eq:var} to $k = \ell$ and $j = a_n +1$, we obtain
\begin{equation*}
	\abs{\m{E}^\omega_{n}\left[v^*(\omega_{(\ell+1) a_n+1})\right] - v^*(\omega_{\ell a_n+1})} \leq p_n^{-2}.
\end{equation*}
Let $j = m-k'a_n$. Because $j \leq a_n$ and $k' \leq p_n - 1$, applying \CRef{eq:var} yields:
\begin{equation*}
	\abs{\m{E}^\omega_{n}\left[v^*(\omega_{m})\right] - v^*(\omega_{k' a_n+1})} \leq p_n^{-2}.
\end{equation*}
As a result, 
\begin{equation*}
	\abs{\m{E}^\omega_{n}\left[v^*(\omega_{m})\right] - v^*(\omega)} \leq 2p_n^{-1}.
\end{equation*}

\Step{Proof of the constant payoff property}
\\
Let $t \in \intoo{0}{1}$. Because $a_n/n$ tends to 0, for $n \geq n_0$ large enough, we have $a_n  \leq n-\ceil{tn}$, hence $p_n a_n \geq n-a_n \geq \ceil{tn}$. Using \CRef{eq:omegam}, we get
$$|\E_{\sigma_n, \rho_n}^\omega\pa{v^*(\omega_{\ceil{tn}+1}) - v^*(\omega)}| \leq 2 p_n^{-1}.$$ 
This implies that 
$|\E_{\sigma_n, \rho_n}^\omega\pa{v^*(\omega_{\ceil{tn}+1}) - v^*(\omega)}| \ttn 0$. 
Finally, \CRef{CP-PropBrunoAsympt} proves the constant payoff property.
\end{proof}

\section{Perspectives} \label{sec:perspectives}
In the single-player case, the constant payoff property remains valid even when the state space and action sets are infinite \cite{SVV10}, provided that $(v_n)$ converges uniformly. A natural extension of our result would be to address two-player stochastic games with infinite action sets, infinite state space and/or imperfect observation of the state, for which $(v_n)$ converges uniformly. Examples of such classes can be found in \cite{SZ16,ziliotto2021,LS15}. To date, the only positive result in this direction pertains to discounted absorbing games with compact action sets \cite{SV20}, and for a particular class of stochastic games with compact state space called \textit{splitting game} \cite{OB18}. 
\\
Another direction is to investigate whether the constant payoff property holds for \textit{any} asymptotically optimal profile. 
\section*{Acknowledgments}
This work was supported by the French Agence Nationale de la Recherche (ANR) under grant ANR-21-CE40-0020 (CONVERGENCE project), and under grant ANR-17-EURE-0010 (Investissements d'Avenir program).
Part of this work was done during a 1-year visit of Bruno Ziliotto to the Center for Mathematical Modeling (CMM) at University of Chile in 2023, under the IRL program of CNRS. The authors are grateful to Rida Laraki and Guillaume Vigeral for fruitful discussions. 

\bibliography{bibliogen.bib}
\end{document}